\documentclass[11pt,reqno,a4paper,twoside]{article}

\usepackage{amsmath,amsthm,amstext,amscd,amssymb,url}
\usepackage{euscript}
\usepackage{graphicx,multicol}

\usepackage{amssymb} 
\usepackage{amsmath}
\usepackage{textcomp}
\usepackage{lmodern}
\usepackage[latin1]{inputenc}

\usepackage{mathrsfs}
\usepackage{epsfig}
\usepackage[inline]{enumitem}

\newcommand{\R}{\mathbb R}
\newcommand{\N}{\mathbb N}
\newcommand{\C}{\mathbb C}

\def\1{{\mathchoice {\rm 1\mskip-4mu l} {\rm 1\mskip-4mu l}
{\rm 1\mskip-4.5mu l} {\rm 1\mskip-5mu l}}}

\newtheorem{theorem}{{\bf{\small T}}{\scriptsize HEOREM}}[section]
\newtheorem{corollary}{{\bf{\small C}{\scriptsize OROLLARY}}}[section]
\newtheorem{proposition}{{\bf{\small P}{\scriptsize ROPOSITION}}}[section]
\newtheorem{lemma}{{\bf{\small L}{\scriptsize EMMA}}}[section]
\newtheorem{remark}{{\bf{\small R}{\scriptsize EMARK}}}[section]
\newtheorem{definition}{{\bf{\small D}{\scriptsize EFINITION}}}[section]
\theoremstyle{definition}

\newtheorem{example}{{\bf{\small E}}{\scriptsize XAMPLE}}[section]

\newcommand{\beq}{\begin{eqnarray}}
\newcommand{\eeq}{\end{eqnarray}}

\newcommand{\ba}{\begin{align*}}
\newcommand{\ea}{\end{align*}}

\newcommand{\be}{\begin{equation}}
\newcommand{\ee}{\end{equation}}

\newcommand{\bl}{\begin{lemma}}
\newcommand{\el}{\end{lemma}}

\newcommand{\br}{\begin{remark}}
\newcommand{\er}{\end{remark}}

\newcommand{\bt}{\begin{theorem}}
\newcommand{\et}{\end{theorem}}

\newcommand{\bd}{\begin{definition}}
\newcommand{\ed}{\end{definition}}

\newcommand{\bp}{\begin{proposition}}
\newcommand{\ep}{\end{proposition}}

\newcommand{\bc}{\begin{corollary}}
\newcommand{\ec}{\end{corollary}}

\newcommand{\bi}{\begin{itemize}}
\newcommand{\ei}{\end{itemize}}

\newcommand{\ben}{\begin{enumerate}}
\newcommand{\een}{\end{enumerate}}

\newcommand{\SEP}{\text{\normalfont SEP}}
\newcommand{\ASEP}{\text{\normalfont ASEP}}

\newcommand{\eft}{\text{\normalfont \tiny left}}
\newcommand{\ight}{\text{\normalfont \tiny right}}
\newcommand{\cheap}{\text{\normalfont \tiny cheap}}

\newcommand{\s}{\text{\normalfont S}}

\newmuskip\pFqmuskip

\newcommand*\pFq[6][8]{%
  \begingroup 
  \pFqmuskip=#1mu\relax
  \mathcode`\,=\string"8000
  \begingroup\lccode`\~=`\,
  \lowercase{\endgroup\let~}\pFqcomma
  {}_{#2}F_{#3}{\left[\genfrac..{0pt}{}{#4}{#5};#6\right]}%
  \endgroup
}
\newcommand{\pFqcomma}{\mskip\pFqmuskip}

\makeatletter
\def\Ddots{\mathinner{\mkern1mu\raise\p@
		\vbox{\kern7\p@\hbox{.}}\mkern2mu
		\raise4\p@\hbox{.}\mkern2mu\raise7\p@\hbox{.}\mkern1mu}}
\makeatother

\begin{document}
\title{{\bf Stochastic duality and eigenfunctions}}
\author{Frank Redig, Federico Sau\\	
\small{Delft Institute of Applied Mathematics}\\
\small{Delft University of Technology}\\
{\small van Mourik Broekmanweg 6, 2628 XE}\\ \small{Delft, The Netherlands}\\
\medskip
\url{f.h.j.redig@tudelft.nl}, \url{f.sau@tudelft.nl}\\}

\maketitle
\begin{abstract}
	We start from the observation that, anytime two Markov generators share an eigenvalue, the function constructed from the product of the two eigenfunctions associated to this common eigenvalue is a duality function. We push further this observation and provide a full characterization of duality relations in terms of spectral decompositions of the generators for finite state space Markov processes. Moreover, we study and revisit some well-known instances of duality, such as Siegmund duality, and extract spectral information from it. Next, we use the same formalism to construct all duality functions  for some solvable examples, i.e., processes for which the eigenfunctions of the generator are explicitly known. 
	
\end{abstract}

\section{Introduction}

Stochastic duality is a technique to connect two Markov  processes via a so-called \emph{duality function}. This connection, interesting in its own right, turns out to be extremely useful when the  \emph{dual} process is more tractable than the original process.

Several applications of stochastic duality may be found  in the context of interacting particle systems \cite{liggett} as, for instance, in the study of hydrodynamic limits and fluctuations  \cite{corwinkpz}, \cite{demasi}, \cite{kipnislandim}, characterization of extremal measures \cite{liggett}, \cite{d1method}, derivation of Fourier law of transport  \cite{dualitytransport}, \cite{kmp} and correlation inequalities \cite{kurchan}. Other fields rich of applications are population genetics, where the coalescent process arises as a natural dual process  (see \cite{greven} and references therein) and branching-coalescing processes \cite{etheridge}.  Duality and related notions have already been used in the study of spectral gaps and convergence to stationarity by several authors, see e.g.\ \cite{choipatie}, \cite{diaconisfill}, \cite{fill}, \cite{miclo}, \cite{patie}.

Part of the research about stochastic duality deals with the problem of \emph{finding} and \emph{characterizing} duality functions relating two given Markov processes.
This means that, for a given pair of Markov generators, one wants to find all duality functions  or, alternatively, a basis of the linear space of duality functions. 
See, for instance, in this direction \cite{mohle} in the context of population genetics, while for particle systems the works  \cite{corwin}, \cite{carincipop}, \cite{chiara}, \cite{d1method}  for symmetric  and \cite{carinci}, \cite{carinci2}, \cite{schutz} for asymmetric processes. For Markov processes, algebraic constructions of duality relations for specific classes of models have also been provided (see e.g.\ \cite{corwin}, \cite{carinci}, \cite{chiara2}, \cite{wolter}, \cite{kuan}). 

In this paper we first show that, viewing a duality relation as a \emph{spectral} relation among the associated Markov generators,  duality functions can be obtained from linear combinations of products of eigenfunctions associated to a common eigenvalue. Secondly, we establish this connection   with the general aim of characterizing all possible dualities in terms of eigenfunctions and generalized eigenfunctions of the generators involved. To this purpose, our discussion mainly focuses on continuous-time finite-state Markov chains for which no reversibility is assumed but canonical eigendecompositions of Jordan-type of the generators are available. 

We emphasize that this connection between duality and eigenfunctions goes both ways: not only eigenfunctions of a shared spectrum give rise to duality functions, but also the existence of duality relations carries information about the spectrum of the generators.
Here we can already see a clear distinction between the notion of   \emph{self-duality} and \emph{integrability}: knowing certain linear combinations of products of eigenfunctions (self-duality) rather than knowing the eigenfunctions themselves (integrability).  

The rest of the paper is organized as follows. In Section 2 we provide all preliminary notions of stochastic duality for continuous-time Markov chains. After an introductory study of self-duality and duality in the reversible setting in Sections \ref{section self-dual rev} and \ref{section dual rev}, in Section \ref{section nonreversibility}, via Jordan canonical decompositions, we make precise to which extent spectrum and eigenstructure of  generators in duality are shared. In fact, the assumed orthonormality of the eigenfunctions in Sections \ref{section self-dual rev} and \ref{section dual rev} has the only role of simplifying the exposition at a first reading. There, products of orthonormal eigenfunctions are a natural tensor basis w.r.t.\ which express duality functions; this fact allows a direct description of the linear subspace of duality functions in terms of this tensor basis.
In Section \ref{section nonreversibility}, we show how, by dropping reversibility of the generators  and thus orthonormality of the associated eigenfunctions, a tensor basis in terms of product of generalized eigenfunctions is always possible.

We  further investigate the connection between eigenfunctions and particular instances of dualities that typically appear in the context of interacting particle systems, see e.g.\  \cite{chiara}, \cite{d1method}, in Sections \ref{section self-dual rev} and  \ref{section dual rev}.
In Section \ref{subsection intertwining} we revisit the notion of \emph{intertwining} (see e.g.\ \cite{huilletmartinez}) in this setting and provide an application to the symmetric exclusion process in Section \ref{subsection exclusion}. In Section \ref{section Siegmund} we provide an alternative way of proving and characterizing Siegmund duality \cite{huilletmartinez}, \cite{siegmund} in the finite context.

\section{Setting and notation}\label{section setting}
Let $\Omega$ be a finite state space with cardinality $|\Omega|=n$. We consider an \emph{irreducible} continuous-time Markov process $\{X_t,\ t \geq 0\}$ on $\Omega$, with generator $L$ given by
\beq \nonumber
L f(x) &=& \sum_{y \in \Omega} \ell(x,y)(f(y)-f(x))\ ,
\eeq
where $f : \Omega \to \R$ is a real-valued function and $\ell : \Omega \times \Omega \to [0,+\infty)$ gives the transition rates. For $x \in \Omega$, we define the exit rate from $x \in \Omega$ as
\beq \nonumber
\ell(x) &=& \sum_{y \in \Omega \setminus \{x\}} \ell(x,y)\ .
\eeq
In the finite context we can identify $L$ with the matrix, still denoted by $L$, given by
\beq \nonumber
L(x,y)\ =\ \ell(x,y)\ \text{ for }\ x \neq y\ ,\quad L(x,x)\ =\ -\ell(x)\ .
\eeq
Given two state spaces $\Omega$, $\widehat \Omega$ of cardinalities $|\Omega|=n$, $|\widehat \Omega|=\widehat n$, and two \emph{Markov processes} with generators $L$, $\widehat L$, we say that they are \emph{dual} with \emph{duality function} $D : \widehat \Omega \times \Omega \to \R$ if, for all $x \in \Omega$ and $\widehat x \in \widehat \Omega$, we have
\beq \label{dualrel}
\widehat L_\eft D(\widehat x,x) &=& L_\ight D(\widehat x, x)\ ,
\eeq 
where \textquotedblleft left\textquotedblright, resp.\ \textquotedblleft right\textquotedblright, refers to action on the left, resp.\ right, variable. If the laws of the two processes coincide, we speak about \emph{self-duality}. The same notion in terms of matrix multiplication, where $D$ also denotes the matrix with entries $\{D(\widehat x,x),\ \widehat x \in \widehat \Omega,\ x \in \Omega \}$, is expressed as
\beq \nonumber
\sum_{\widehat y \in \widehat \Omega} \widehat L(\widehat x, \widehat y) D(\widehat y, x) &=& \sum_{y \in \Omega} L(x,y)D(\widehat x, y) \ ,
\eeq
or, shortly, as
\beq \label{dualrel matrix}
\widehat L D &=& D L^\mathsf{T}\ ,
\eeq
where the symbol $^\mathsf{T}$ denotes \emph{matrix  transposition}, i.e., for a matrix $A$, 
\beq \nonumber
(A^\mathsf{T})(x,y)&=& A(y,x)\ ,\quad x,y \in \Omega\ .
\eeq
More generally, we define two \emph{operators} $\widehat L$ and $L$ \emph{dual} with duality function $D$ if relation \eqref{dualrel}, or equivalently \eqref{dualrel matrix} in matrix notation, holds. 


%

\section{Self-duality from eigenfunctions:  reversible case}\label{section self-dual rev}
As in Section \ref{section setting}, let $\Omega$ be a finite set of cardinality $|\Omega|=n$, and let $L$ be a generator of an irreducible \emph{reversible} Markov process
on $\Omega$ w.r.t.\ the positive measure $\mu$. This measure then satisfies the detailed balance condition
\be\label{detbal}
\mu(x) L(x,y)= \mu(y) L(y,x)\ ,
\ee
for all $x,y\in \Omega$. This relation can be rewritten as a self-duality with self-duality function the so-called
\emph{cheap self-duality function}:
\be\label{dcheap}
D_{\cheap} (x,y)= \frac{\delta_{x,y}}{\mu(y)}\ .
\ee
The reversibility of $\mu$ implies that $L$ is self-adjoint in $L^2(\mu)$ and, as a consequence, there exists a basis $\{ u_1, \ldots, u_n\}$ of eigenfunctions of $L$ with
$u_1(x)=1/\sqrt{n}$ corresponding to eigenvalue zero and $\{ u_1, \ldots, u_n\}$ orthonormal, i.e.,
$\langle u_i, u_j\rangle_\mu=\delta_{i,j}$ where $\langle \cdot,\cdot \rangle_\mu$ denotes inner product in $L^2(\mu)$.
We denote by $\{\lambda_1,\ldots, \lambda_n \}$ the corresponding real eigenvalues with
\beq \nonumber
0\ =\ \lambda_1\ > \lambda_2\ \geq\ \ldots\ \geq\ \lambda_n\ .
\eeq
The following proposition then shows how to obtain and characterize self-duality functions in terms of this orthonormal system.
The last statement recovers an earlier result from \cite{chiara2}.
\begin{proposition}\label{simpelprop}
\ben[label={(\roman*)}]
\item For $a_1, a_2\ldots, a_n\in\R$, the function
\beq\label{tensdual}
D(x,y)&=&\sum_{i=1}^n a_i u_i(x) u_i(y)
\eeq
is a self-duality function.
\item Every self-duality function has a unique decomposition of the form
\beq\label{genform}
D(x,y) &=& \sum_{i,j : \lambda_i=\lambda_j} a_{ij} u_i(x) u_j(y)\ .
\eeq
\item If a function of the form $D(x,y) =  f(x)g(y)$ is a non-zero self-duality function, then $f$ and $g$ are eigenfunctions corresponding to the same eigenvalue.
\item The $L^2(\mu)$ inner product of self-duality functions produces self-duality functions, i.e.,
if $D$ and $D'$ are self-duality functions,
then
\beq \label{inner product}
\langle D(x,\cdot), D'(x',\cdot)\rangle_\mu &=& D''(x,x')
\eeq
defines a self-duality function $D''$.
\een
\end{proposition}
\begin{proof}
For \emph{(i)}, by definition of eigenfunction $L u_i = \lambda_i u_i$ with $\lambda_i \in \R$, we obtain
\begin{multline*} \nonumber
L_\eft D(x,y)\ =\ \sum_{i=1}^n a_i Lu_i(x) u_i(y)\ =\ \sum_{i=1}^n a_i \lambda_i u_i(x) u_i(y)\\
\nonumber\ =\ \sum_{i=1}^n a_i u_i(x)\lambda_i u_i(y)\ =\ \sum_{i=1}^n a_i u_i(x) Lu_i(y)\ =\ L_\ight D(x,y)\ ,
\end{multline*}
hence \eqref{dualrel}.

For \emph{(ii)}, start by noticing that every function $D:\Omega\times\Omega\to\R$ can be written in a unique way as
\beq \nonumber
D(x,y) &=& \sum_{i,j=1}^n a_{i,j} u_i(x) u_j(y)\ ,
\eeq
Now using the duality relation \eqref{dualrel}, it follows that
\beq \nonumber
\sum_{i,j} a_{i,j} \lambda_i u_i(x) u_j(y) &=& \sum_{i,j} a_{i,j} \lambda_j u_i(x) u_j(y)\ ,
\eeq
which implies that, for all $i,j=1,\ldots,n$,
\beq \nonumber
a_{i,j}\lambda_i\ =\ a_{i,j} \lambda_j\ .
\eeq

For item \emph{(iii)}, first write
\beq \nonumber
f(x) g(y) &=&\sum_{i,j=1}^n a_{ij} u_i(x) u_j(y)\ .
\eeq
Then we find $a_{ij} = \langle f, u_i \rangle_\mu \langle g, u_j\rangle_\mu=: \alpha_i \beta_j$. From self-duality we conclude, for all $i,j=1,\ldots, n$,
\beq \nonumber
\alpha_i \beta_j (\lambda_i-\lambda_j)&=&0\ .
\eeq
Now use that $f(x)g(y)$ is not identically zero to conclude that there exists $i$ with $\alpha_i\not=0$.
Then if $\lambda_j\not=\lambda_i$ we conclude $\beta_j=0$, which implies that $g$ is an eigenfunction
with eigenvalue $\lambda_i$. Because $g$ is not identically zero, we can reverse the argument and conclude.

For \emph{(iv)}, by exchanging the order of summations and using $\langle u_j, u_l \rangle_\mu = \delta_{j,l}$, the l.h.s.\ of \eqref{inner product} reads
\begin{small}
	\begin{multline*} \nonumber
	\sum_{y \in \Omega} D(x,y) D(x',y) \mu(y)\\ =\ \sum_{y \in \Omega} \left(\sum_{i,j: \lambda_i=\lambda_j} a_{i,j} u_i(x) u_j(y)\right) \left( \sum_{k,l: \lambda_k = \lambda_l} a_{k,l} u_k(x') u_l(y) \right) \mu(y)\\
	=\ \sum_{j=1}^n \left(\sum_{i: \lambda_i=\lambda_j} a_{i,j} u_i(x)\right) \left(\sum_{k: \lambda_k = \lambda_j} a_{k,j} u_k(x') \right)\ .
	\end{multline*}
\end{small}
By noting that, for all $j=1,\ldots, n$,  the function $u'_j = \sum_{i: \lambda_i=\lambda_j} a_{i,j} u_i$ is either vanishing or is an eigenfunction of $L$ associated to $\lambda_j$, the proof is concluded.

\end{proof}

In the next propositions we study particular instances of self-duality functions. More precisely, by using Proposition \ref{simpelprop}, we recover the
cheap self-duality function in \eqref{dcheap}, while in Proposition \ref{osf} we characterize \emph{orthogonal} self-duality functions (cf.\  \eqref{definition orthogonal}--\eqref{definition orthogonal 2} below).

\begin{proposition}[Cheap self-duality]\label{proposition cheaprev}
\begin{enumerate}[label={(\roman*)}]
	\item 
For the choice $a_1=a_2=\ldots =a_n=1$ in \eqref{tensdual}, we obtain the cheap self-duality function, i.e.,
\beq \label{cheaprev}
D_{\cheap}(x,y)\ =\ \frac{\delta_{x,y}}{\mu(y)}\ =\ \sum_{i=1}^n u_i(x) u_i(y)\ .
\eeq
\item  Conversely, if $\{v_1,\ldots, v_n \}$ is a basis of $L^2(\mu)$ and satisfies
\beq \label{v functions}
\sum_{i=1}^n v_i(x)v_i(y) &=& \frac{\delta_{x,y}}{\mu(y)}
\eeq
for all $x, y \in \Omega$,
then $\{v_1,\ldots, v_n \}$ is an orthonormal basis of $L^2(\mu)$.
\end{enumerate}
\end{proposition}
\begin{proof}
To show \eqref{cheaprev}, by the positivity of $\mu$, we need to show that, for all $f:\Omega\to\R$ and $x\in\Omega$,
\[
\sum_{y \in \Omega} \sum_{i=1}^n u_i(x) u_i(y) \mu(y) f(y)\ =\ f(x)\ .
\]
Now note, by interchanging the sum over $i$ with the sum over $y$, that the l.h.s.\ equals
\[
\sum_{i=1}^n u_i(x) \langle u_i, f\rangle_{\mu}\ =\ f(x)\ ,
\]
and hence we obtain \emph{(i)}.

For \emph{(ii)} we need to show that for all $f : \Omega \to \R$ and $x \in \Omega$
\beq \label{equality}
f(x)\ =\ \sum_{i=1}^n v_i(x) \langle v_i, f \rangle_\mu\ =\ \sum_{i=1}^n \sum_{y \in \Omega} v_i(x) v_i(y) f(y) \mu(y)\ .
\eeq
We conclude by interchanging the order of the two summations in the r.h.s.\ above and using \eqref{v functions}, we indeed obtain \eqref{equality}.
\end{proof}

Remark that the cheap self-duality function is the only, up to multiplicative constants, \emph{diagonal} self-duality, and that it  is \emph{orthogonal} in the sense that, for all $x, x' \in \Omega$,
\beq \label{definition orthogonal}
\langle D_{\cheap} (x,\cdot), D_{\cheap} (x', \cdot)\rangle_\mu &=& \delta_{x,x'}\  \langle D_{\cheap} (x,\cdot), D_{\cheap} (x, \cdot)\rangle_\mu\ ,
\eeq
and similarly, for all $y, y' \in \Omega$,
\beq \label{definition orthogonal 2}
\langle D_{\cheap} (\cdot, y), D_{\cheap} (\cdot, y')\rangle_\mu &=& \delta_{y,y'}\  \langle D_{\cheap} (\cdot,y), D_{\cheap} ( \cdot,y)\rangle_\mu\ .
\eeq
The next proposition shows how to find \emph{all} orthogonal self-duality functions.
\begin{proposition}[Orthogonal self-duality]\label{osf}
\ben[label={(\roman*)}]
\item If $\{\tilde u_1, \ldots, \tilde u_n\}$ is an orthonormal system  in $L^2(\mu)$ of eigenfunctions of $L$, corresponding
to the same eigenvalues $\{\lambda_1,  \ldots,\lambda_n\}$, then
\beq\label{orthodual}
D(x,y)&=& \sum_{i=1}^n  \tilde{u}_i(x) u_i(y)
\eeq
is an orthogonal self-duality function. More precisely, for all $x, x' \in \Omega$,
\beq \label{orthogonal selfdual}
\langle D(x, \cdot), D(x', \cdot)\rangle_\mu&=& \frac{\delta_{x,x'}}{\mu(x')}\ .
\eeq
\item The self-duality functions of the form \eqref{orthodual} are the only, up to a multiplicative factor,  orthogonal self-duality functions. 
\een
\end{proposition}
\begin{proof}
For \emph{(i)}, we compute, for all $k=1,\ldots, n$ and $x \in \Omega$, the following quantity
\beq \nonumber
\sum_{x' \in \Omega} \langle D(x,\cdot), D(x',\cdot) \rangle_\mu \tilde u_k(x') \mu(x')\ .
\eeq
By $\langle u_i, u_j \rangle_\mu = \langle \tilde u_i, \tilde u_j \rangle_\mu = \delta_{i,j}$, the line above rewrites as follows:
\beq \nonumber
&&\sum_{x' \in \Omega} \sum_{y \in \Omega} \left(\sum_{i=1}^n \tilde u_i(x) u_i(y) \right)\left(\sum_{j=1}^n \tilde u_j(x') u_j(y) \right) \mu(y) \tilde u_k(x') \mu(x') \\
\nonumber
&=& \sum_{i=1}^n \sum_{j=1}^n \tilde u_i(x) \left(\sum_{y \in \Omega} u_i(y) u_j(y) \mu(y) \right) \left(\sum_{x' \in \Omega} \tilde u_j(x') \tilde u_k(x') \mu(x') \right)\\
\nonumber
&=& \sum_{i=1}^n \sum_{j=1}^n \tilde u_i(x) \delta_{i,j} \delta_{j,k}\ =\ \tilde u_k(x)\ .
\eeq
This together with Proposition \ref{proposition cheaprev} concludes the proof of part \emph{(i)}.

For \emph{(ii)}, by starting from a general self-duality function
\beq \nonumber
D(x,y) &=& \sum_{i,j:\lambda_i=\lambda_j} a_{i,j} u_i(x) u_j(y)\ ,
\eeq
the l.h.s.\ of \eqref{orthogonal selfdual} rewrites  as 
\beq \nonumber
\sum_{j =1}^n u'_j(x) u'_j(x')\ ,
\eeq
where $\{u'_1, \ldots, u'_n \}$ is  defined as 
\beq \nonumber
u'_j(x) &=& \sum_{i: \lambda_i=\lambda_j} a_{i,j} u_i(x)\ .
\eeq
By remarking that either $u'_j = 0$ or $u'_j$ is an eigenfunction of $L$ associated to $\lambda_j$ and applying
  Proposition \ref{proposition cheaprev},  we have  that
\beq \nonumber
\langle u'_i, u'_j \rangle_\mu &=& \delta_{i,j}\ ,
\eeq
and that the self-duality function $D$ has the form \eqref{orthodual} with $\tilde u_i = u'_i$.
\end{proof}

\section{Duality from eigenfunctions:  reversible case}\label{section dual rev}
Now we consider two generators $L$, $\widehat{L}$ on the same finite state space $\Omega$ with reversible measures $\mu$, $\widehat{\mu}$ respectively, 
and orthonormal systems of eigenfunctions $\{u_1,\ldots, u_n\}$, $\{\widehat{u}_1,\ldots, \widehat{u}_n\}$ corresponding to the {\em same} real eigenvalues
$\{\lambda_1,\ldots, \lambda_n\}$, i.e., we assume that $L$ and $\widehat L$ are self-adjoint in $L^2(\mu)$, resp.\ in $L^2(\widehat \mu)$, and that they are iso-spectral. 

In what follows we state - without proofs - analogous relations between duality functions and orthonormal systems of eigenfunctions of $L$ and $\widehat L$.
\begin{proposition} \label{proposition duality rev}
\ben[label={(\roman*)}]
\item For $a_1,\ldots, a_n\in\R$ the function
\beq \nonumber
D(\widehat x,x)&=&\sum_{i=1}^n a_i \widehat{u}_i(\widehat x) u_i(x)
\eeq
is a duality function for duality between $\hat{L}$ and $L$.
\item Every duality function has  a unique decomposition  of the form
\beq \nonumber
D(\widehat x,x)&=& \sum_{i,j: \lambda_i=\lambda_j} a_{ij} \widehat{u}_i(\widehat x) u_j(x)\ .
\eeq
\item If a function of the form $D(\widehat x ,x)= f(\widehat x )g(x)$ is a non-zero duality function, then
$f$  and $g$ are eigenfunctions of $\widehat L$, resp.\ $L$, corresponding to the same eigenvalue.
\item The $L^2(\mu)$ and $L^2(\widehat \mu)$ inner products of duality functions produce self-duality functions, i.e.,
if $D$ and $D'$ are duality functions,
then
\beq \nonumber
\langle D(\widehat x,\cdot), D'(\widehat x',\cdot)\rangle_\mu &=& \widehat D(\widehat x,\widehat x')
\eeq
defines a self-duality function $\widehat D$ for $\widehat L$, and similarly
\beq \nonumber
\langle D(\cdot, x), D'(\cdot, x') \rangle_{\widehat \mu} &=& \widetilde D(x,x')
\eeq
determines  a self-duality function $\widetilde D$ for $L$.

\een
\end{proposition}

\begin{proposition}[Orthogonal duality]
\ben[label={(\roman*)}]
\item If $\{\tilde u_1,\ldots,\tilde u_n\}$ is an orthonormal system in $L^2(\widehat \mu)$ of eigenfunctions of $\widehat L$ corresponding
to the same eigenvalues $\{\lambda_1,\ldots, \lambda_n\}$, then
\be
D(\widehat x,x)= \sum_{i=1}^n  \tilde{u}_i(\widehat x) u_i(x)
\ee
is an orthogonal duality function, i.e.,
\beq \nonumber
\langle D(\widehat x, \cdot), D(\widehat x', \cdot)\rangle_\mu&=& \frac{\delta_{\widehat x,\widehat x'}}{\widehat{\mu}(\widehat x')}
\eeq
and
\beq \nonumber
\langle D(\cdot, x), D(\cdot, x')\rangle_{\widehat{\mu}}&=& \frac{\delta_{x,x'}}{\mu(x')}\ .
\eeq
\item These are the only, up to multiplicative constants, orthogonal dualities between $\widehat{L}$ and $L$.
\een
\end{proposition}

\section{Duality from eigenfunctions:  non-reversible case}\label{section nonreversibility}
Working in  the \emph{non-reversible} context, i.e., whenever there \emph{does not} exist a probability measure $\mu$ on $\Omega$ for which the generator $L$ is self-adjoint in $L^2(\mu)$, a spectral decomposition of the generator in terms of real non-positive eigenvalues and orthonormal real eigenfunctions is typically lost. In recent years, the study of the eigendecomposition of non-reversible generators has received an increasing attention \cite{choipatie}, \cite{choipatie2}, \cite{weber2}, \cite{patie}, \cite{weber} and duality-related notions have been introduced to relate spectral information of one process, typically a reversible one, to another, typically non-reversible \cite{fill}, \cite{miclo}. 

However, regardless of the spectral eigendecomposition of the generators, in principle interesting dualities can still be constructed from eigenfunctions, either real or complex, and generalized eigenfunctions of the generators involved. The key on which this relation builds up, in the finite context, is the \emph{Jordan canonical decomposition} of the generators. A relation
between duality and the Jordan canonical decomposition has already been used in the context of models of population dynamics in \cite{mohle}.

Below, before studying the most general result that exploits the Jordan form of the generators, we treat some special cases reminiscent of the previous sections. In the sequel, for a function $u : \Omega \to \C$, we denote by $u^\ast : \Omega \to \C$ its complex conjugate. 

\subsection{Duality from complex eigenfunctions}
A first feature that typically drops as soon as one moves to the non-reversible situation is the appearance of only real eigenvalues. Indeed, given a non-reversible generator $L$ of an irreducible Markov process on  $\Omega$, pairs of complex conjugates eigenvalues $\{\lambda, \lambda^\ast \}$ and eigenfunctions $\{u, u^\ast \}$ may arise as in the following example.
\begin{example}
	The continuous-time Markov chain on the state space $\Omega=\{1,2,3 \}$ and described by the generator $L$, which, viewed as a matrix, reads 
	\beq \nonumber
	L &=& \begin{pmatrix}
		-1&1 &0\\
		0& -1&1\\
		1&0&-1 
	\end{pmatrix}\ ,
	\eeq
	represents a basic example of this situation. Indeed, the Markov chain is irreducible, the eigenvalues $\{\lambda_1,\lambda_2, \lambda_3 \}$ are
	\beq \nonumber
	\lambda_1\ =\ 0\ ,\quad \lambda_2\ =\ \lambda_3^\ast\ =\ -\frac{3}{2} + i \frac{\sqrt{3}}{2}\ ,
	\eeq
	while the associated eigenfunctions $\{u_1,u_2, u_3 \}$ are, for $x \in \{1,2,3 \}$,
	\beq \nonumber
	u_1(x)\ =\ \frac{1}{\sqrt{3}}\ ,\quad u_2(x)\ =\ u_3^\ast(x)\ =\  e^{(i  \frac{2}{3} \pi) x}\ .
	\eeq
	\qedsymbol
\end{example}
Let us, thus, consider two irreducible non-reversible generators $L$, $\widehat L$ on the same state space $\Omega$. We investigate the situation in which there exist $\lambda \in \C \setminus \R$ and functions $u, \widehat u :  \Omega \to \C$ such that
\beq \label{eigenfct complex}
L u\ =\ \lambda u\ ,\quad \quad \widehat L \widehat u\ =\ \lambda \widehat u\ .
\eeq
Remark that, as $L$, $\widehat L$ are real operators, this implies that
\beq \label{eigenfct complex 2}
L u^*\ =\ \lambda^* u^*\ ,\quad \quad \widehat L \widehat u^*\ =\ \lambda^* \widehat u^*\ .
\eeq

A real duality function arising from a shared pair of complex eigenvalues is obtained in the following proposition. 
\begin{proposition}\label{proposition complex}
 For $a \in \R$, the  function 
\beq \nonumber
D(\widehat x, x) &=& a\widehat u(\widehat x) u(x) + a \widehat u^*(\widehat x) u^\ast(x) 
\eeq
takes values in $\R$ and
is a duality function between $\widehat L$ and $L$. 
\end{proposition}
\begin{proof}
	It is clear that $D(\widehat x,x)$ is in $\R$. 
	Then, by using \eqref{eigenfct complex} and \eqref{eigenfct complex 2}, we obtain
	\begin{multline*} \nonumber
	\widehat L_\eft D(\widehat x, x)\ =\ a (\widehat L \widehat u)(\widehat x) u(x) + a(\widehat L \widehat u^\ast)(\widehat x) u^\ast(x)\\ =\ a \lambda \widehat u(\widehat x) u(x) + a \lambda^\ast \widehat u^\ast(\widehat x)u^\ast(x)\ =\  a\widehat u(\widehat x) \lambda u(x) + a \widehat u^\ast(\widehat x) \lambda^\ast u^\ast(x)\\ =\ a \widehat u(\widehat x) (L u)(x) + a \widehat u^\ast(\widehat x) (L u^\ast)(x) \ =\ L_\ight D(\widehat x, x)\ .
	\end{multline*}
\end{proof}
\subsection{Duality from generalized eigenfunctions}
A second feature that may be lacking is the existence of a linear independent system of eigenfunctions. However, if $L$ is an irreducible non-reversible generator on the state space $\Omega$ with real non-negative eigenvalues $\{\lambda_1,\ldots, \lambda_n \}$, there always exists a linearly independent system of so-called \emph{generalized eigenfunctions}, i.e., for each eigenvalue $\lambda_i$, there exists a set of linearly independent functions $\{u_i^{(1)},\ldots, u_i^{(m_i)} \}$ such that $m_i \leq n$, 
\beq \nonumber
L u_i^{(1)} &=& \lambda_i u_i^{(1)}
\eeq
and, for $1 < k \leq m_i$,
\beq \nonumber
L u_i^{(k)} = \lambda_i u_i^{(k)} + u_i^{(k-1)}\ .
\eeq 
We refer to $u_i^{(k)}$ as the \emph{$k$-th order generalized eigenfunction} associated to $\lambda_i$.
Moreover, if $\lambda_i \neq \lambda_j$, then the set $\{u_i^{(1)},\ldots, u_i^{(m_i)},u_j^{(1)},\ldots, u_j^{(m_j)} \}$ is linearly independent and any arbitrary function $f : \Omega \to \R$ can be written as linear combination of functions in $\{u_i^{(k)},\ i=1,\ldots, n;\ k = 1,\ldots, m_i \}$. 

\begin{example}
The irreducible generator $L$ on the state space $\Omega=\{1,2,3,4 \}$ given by
\beq \nonumber
L &=& \begin{pmatrix}
	-\frac{1}{2}&\frac{1}{2} &0&0\\
	0& -1&\frac{1}{2}&\frac{1}{2}\\
	\frac{1}{2}&0&-1&\frac{1}{2}\\
	0&\frac{1}{2}&\frac{1}{2}&-1
\end{pmatrix}\ ,
\eeq
represents a basic example of this situation. Indeed, the eigenvalue $\lambda=-1$ has $u^{(1)}$ given by
\beq \nonumber
u^{(1)}(x)\ =\ \frac{(-1)^x}{2}\ ,\quad x \in \{1,2,3,4 \}\ ,
\eeq
as eigenfunction and
\beq \nonumber
u^{(2)}(x)\ =\ \cos\left({\frac{\pi}{2} (x+1)}\right)\ ,\quad x \in \{1,2,3,4\}\ ,
\eeq
as a second order generalized eigenfunction, i.e.,
\beq \nonumber
L u^{(2)}\ =\ -u^{(2)} + u^{(1)}\ .
\eeq
\qedsymbol
\end{example}
In this situation, in case of two generators $L$, $\widehat L$ sharing a real eigenvalue $\lambda$ with associated generalized eigenfunctions $\{u^{(1)},\ldots, u^{(m)} \}$, $\{\widehat u^{(1)},\ldots, \widehat u^{(m)} \}$, the main idea is that a duality function is readily constructed from sums of products of generalized eigenfunctions whose order is, nevertheless, reversed. This connection is the content of the following proposition.
\begin{proposition}\label{proposition generalized eigenvectors}
The function
\beq \nonumber
D(\widehat x, x) &=& \sum_{k=1}^m \widehat u^{(k)}(\widehat x) u^{(m+1-k)}(x)
\eeq
is a duality function between $\widehat L$ and $L$.
\end{proposition}
\begin{proof}
By  using the definition of $k$-th order generalized eigenfunction, we obtain
\begin{small}
\begin{multline*} \nonumber
\widehat L_\eft D(\widehat x, x)\ =\ \sum_{k=1}^m (\widehat L \widehat u^{(k)})(\widehat x) u^{(m+1-k)}(x)\\ =\ \sum_{k=1}^m \lambda \widehat u^{(k)}(\widehat x) u^{(m+1-k)} + \sum_{k=2}^m \widehat u^{(k-1)}(\widehat x) u^{(m+1-k)}(x)\\ \nonumber
=\ \sum_{k=1}^m \lambda \widehat u^{(k)}(\widehat x) u^{(m+1-k)} + \sum_{k=1}^{m-1} \widehat u^{(k)}(\widehat x) u^{(m-k)}(x)\\ =\ \sum_{k=1}^m \widehat u^{(k)}(\widehat x) (L u^{(m+1-k)})(x)\ \nonumber
=\ L_\ight D(\widehat x,x)\ .
\end{multline*}
\end{small}
\end{proof}

\subsection{Duality and the Jordan canonical decomposition: general case}
In this section we provide a general framework that allows us to cover all instances of duality encountered so far in the finite setting. The standard strategy of decomposing generators - viewed as matrices - into their Jordan canonical form builds a bridge between dualities and spectral information of the generators involved. In particular, this linear algebraic approach is useful for the problem of  \emph{existence} and \emph{characterization} of duality functions: on one side, the existence of a Jordan canonical decomposition for any generator leads, for instance, to the existence of self-dualities; on the other side,  dualities between generators carry  information about a common, at least partially, spectral structure of the generators.

Before stating the main result, we introduce some notation. Given a generator $L$ on the state space $\Omega$ with cardinality $|\Omega|=n$, $L$ is in \emph{Jordan canonical form} if it can be written as
\beq \nonumber
L &=& U J U^{-1}\ ,
\eeq
where $J \in \C^{n\times n}$ is the \emph{unique}, up to permutations,  \emph{Jordan matrix} \cite[Definition 3.1.1]{Horn} associated to $L$ and 
$U \in \C^{n\times n}$ is an invertible matrix. Recall that columns $\{u_1,\ldots, u_n \}$ of $U$ consists of (possibly generalized)  eigenfunctions of $L$, while the rows $\{w_1,\ldots, w_n \}$ of  $U^{-1}$ the (possibly generalized)  eigenfunctions of $L^\mathsf{T}$, chosen in such a way that
\beq \nonumber
\langle w_i, u_j \rangle &=& \sum_{x \in \Omega} w_i(x) u_i^\ast(x)\ =\ \delta_{i,j}\ .
\eeq
For all Jordan matrices $J \in \C^{n\times n}$ of the form 
\beq \nonumber
J&=&\begin{pmatrix}
	J_{m_1}(\lambda_1) & & \cdots&0\\
	& J_{m_2}(\lambda_n) & &\vdots\\
	\vdots&&\ddots&\\
	0&\cdots&&J_{m_k}(\lambda_k)
\end{pmatrix}\ ,
\eeq
with $m_1 +\ldots + m_k = n$ and \emph{Jordan blocks} $J_m(\lambda)$ of size $m$ associated to eigenvalue $\lambda \in \C$, we define the matrix $B_J \in \R^{n\times n}$ as follows
\beq \nonumber
B_J &=&
\begin{pmatrix}
	H_{m_1} & & \cdots&0\\
	& H_{m_2} & &\vdots\\
	\vdots&&\ddots&\\
	0&\cdots&&H_{m_k}\ 
\end{pmatrix}\ ,
\eeq
where, for all $m \in \N$, the matrix $H_m \in \R^{m\times m}$ is defined as
\beq \nonumber
H_m &=& \begin{pmatrix}
	0&\cdots &&1\\
	\vdots& &\Ddots &\\
	&\Ddots&&\vdots\\
	1&&\cdots&0\ 
\end{pmatrix}\ ,
\eeq
i.e., in such a way that 
$B_J^\mathsf{T} = B_J^{-1} =B_J$
and
$J B_J = B_J J^\mathsf{T}$.
Moreover, we say that two matrices $L \in \R^{n \times n}$, $\widehat L \in \R^{\widehat n \times \widehat n}$ are \emph{$r$-similar} for some $r =1,\ldots, \min\{n, \widehat n\}$ if there exist Jordan canonical forms
\beq \label{jordan forms}
L\ =\ U J U^{-1}\ ,\quad\quad \widehat L\ =\ \widehat U \widehat J \widehat U^{-1}\ ,
\eeq
matrices $S_r \in \R^{\widehat n \times n}$ and $I_r \in \R^{r\times r}$ of the form
\beq \nonumber
S_r\ =\ \begin{pmatrix}
	I_r &\mathbf{0}\\
	\mathbf{0}&\mathbf{0}
	\end{pmatrix}\ ,\quad I_r\ =\ \begin{pmatrix}
	1&  \cdots&0\\
	\vdots&\ddots&\vdots\\
	0&\cdots&1
\end{pmatrix}\ ,
\eeq
and permutation matrices $\widehat P\in \R^{\widehat n \times \widehat n}$ and $P \in \R^{n \times n}$
such that
\beq \nonumber
T_r &=& \widehat P S_r P
\eeq
and
\beq \label{r-similar}
\widehat J T_r &=& T_r J\ .
\eeq
Of course, if two matrices are $r$-similar, then they are necessarily $r'$-similar, for all $r'=1,\ldots, r$ and if $r=n= \widehat n$ then we simply say that they are \emph{similar}.

In the following theorem we establish a general connection between duality relations and Jordan canonical forms for generators $L$, $\widehat L$.

\begin{theorem}\label{theorem jordan}
	The following statements are equivalent:
	\ben[label={(\roman*)}]	\item There exists a duality function $D(\widehat x, x)$ of rank $r$ between $\widehat L$ and $L$\ .
	 \item $L$ and $\widehat L$ are \emph{$r$-similar}\ .	 
	\een
	If either condition holds, any duality function is of the form
	\beq \label{choice D}
	D &=& \widehat U  T_r B_J U^\mathsf{T}\ .
	\eeq
	In particular if $L=\widehat L$, for any $r=1,\ldots, n$, there always exists a self-duality function $D$ of  rank $r$ and it must be of the form \eqref{choice D}.
\end{theorem}
\begin{proof}
We start with proving that \emph{(ii)} implies \emph{(i)}. By using the property of $r$-similarity \eqref{r-similar} with Jordan decompositions as in \eqref{jordan forms}, with the choice \eqref{choice D} of the candidate duality function $D$, we obtain
\beq \nonumber
\widehat L \widehat U T_r B_J U^\mathsf{T}\ =\ \widehat U \widehat J T_r B_J U^\mathsf{T}\ =\ \widehat U T_r J B_J U^\mathsf{T}\ =\ \widehat U T_r B_J J^\mathsf{T} U^\mathsf{T}\ =\ \widehat U T_r B_J U^\mathsf{T} L^\mathsf{T}\ ,
\eeq	
i.e., the duality relation \eqref{dualrel matrix} in matrix form. 

For the other implication, as the matrices $U$, $\widehat U$ in \eqref{jordan forms} and $B_J$ are invertible, the following chains of identities are equivalent:
\beq \nonumber
\widehat L D\ =\ D L^\mathsf{T} &\Longleftrightarrow& \widehat U \widehat J \widehat U^{-1} D\ =\ D (U^{-1})^\mathsf{T} J^\mathsf{T} U^\mathsf{T}\\
\nonumber
&\Longleftrightarrow& \widehat J \widehat U^{-1} D (U^{-1})^\mathsf{T}\ =\ \widehat U^{-1} D (U^{-1})^\mathsf{T} J^\mathsf{T}\\
\nonumber
&\Longleftrightarrow& \widehat J \widehat U^{-1} D (U^{-1})^\mathsf{T} B_J\ =\ \widehat U^{-1} D (U^{-1})^\mathsf{T} B_J J\ .
\eeq
Moreover, if $D$ has rank $r$, then $ \widehat U^{-1} D (U^{-1})^\mathsf{T} B_J$ must have rank $r$ as well. The last relation is of the form
\beq \nonumber
\widehat J A &=& A J\ ,
\eeq
where $A = \widehat U^{-1} D (U^{-1})^\mathsf{T} B_J$ is a matrix of rank $r$. Therefore, we conclude that there exists a permutation matrix $P \in \R^{n \times n}$ such that 
\beq \nonumber
\widehat J S_r  &=& S_r P J P^{-1}\ , 
\eeq
i.e., $L$ and $\widehat L$ are $r$-similar according to the Jordan canonical decompositions
\beq \nonumber
L\ =\ \widetilde U \widetilde J \widetilde U^{-1}\ ,\quad \quad \widehat  L\ =\ \widehat U \widehat J \widehat U^{-1}\ , 
\eeq
with $\widetilde U = U P^{-1}$ and $\widetilde J = P J P^{-1}$. 
\end{proof}
\begin{remark}
	\begin{enumerate}[label={(\alph*)}]
		\item In  words, the theorem above states that there exists a rank-$r$ duality matrix if and only if
		the generators $\widehat L$ and $L$ have $r$ eigenvalues (with multiplicities) in common with \textquotedblleft compatible\textquotedblright\ structure of eigenspaces. Additionally, equation  \eqref{choice D} provides the most general form of the duality function $D$ in terms of matrices $U$, $\widehat U$.  In particular, if $J$ is \emph{diagonal} (i.e., $B_J$ is the identity matrix) \emph{all} duality functions $D(\widehat x, x)$ of rank $r$ read as 
		\beq \nonumber
		D(\widehat x, x) &=& \sum_{i=1}^r a_i \widehat u_i(\widehat x) u_i(x)\ ,
		\eeq
		for $a_1,\ldots, a_n \in \R \setminus \{0\}$, given  $\{u_1,\ldots, u_n \}$, $\{\widehat u_1,\ldots, \widehat u_{\widehat n} \}$ are the columns of $U$, $\widehat U$, invertible matrices in the Jordan decompositions \eqref{jordan forms} satisfying \eqref{r-similar} with   $T_r=S_r$.
		Note the analogy with the duality function described in Propositions \ref{simpelprop}, \ref{proposition duality rev} and \ref{proposition complex}. If $J$ is \emph{non-diagonal}, all duality functions $D$ have a similar form up to some index permutations  as in Proposition \ref{proposition generalized eigenvectors}.

		\item
		We note that the \emph{constant duality function} is always a trivial duality function between any two generators $L$, $\widehat L$ on $\Omega$, $\widehat \Omega$. Indeed, $\lambda = 0$ is always an eigenvalue for both $L$ and $\widehat L$ with associated constant eigenfunctions $u :  \Omega \to \R$, $\widehat u : \widehat \Omega \to \R$, i.e.,  for all $x \in \Omega$ and $\widehat x \in \widehat \Omega$, 
		\beq \nonumber
		u(x)\ =\ 1\ ,\quad \widehat u(\widehat x)\ =\ 1\ ,
		\eeq
		are eigenfunctions for $L$, $\widehat L$ associated to $\lambda = 0$.
		\item 
		Another consequence, as already mentioned in \cite{kurchan2}, is that in the finite context \emph{self-duality} functions \emph{always} exist. In fact, a generator $L$, viewed as a matrix, is  always \emph{similar} to itself. Hence, viewing duality relations between generators as similarity relations among matrices allows one to transfer statements about \emph{existence} of Jordan canonical decompositions to statements regarding the \emph{existence} of duality relations, even when neither any explicit formula of the duality functions nor reversible measures for the processes are known. 
		However,  Theorem \ref{theorem jordan} above provides information on how to construct any self-duality matrix.  Indeed, given any two Jordan decompositions of $L$, say
		\beq \nonumber
		L U\ =\ U J\ ,\quad\quad L \widetilde U\ =\ \widetilde U J\ ,
		\eeq
		the matrix $D$ constructed from $U, \widetilde U$ and $J$ as in \eqref{choice D}, namely
		\beq \label{jordan form 2}
		D\ =\ \widetilde U B_J U^\mathsf{T}\ ,
		\eeq
		turns out to be a self-duality function for $L$ and, viceversa, any self-duality matrix $D$ for $L$ is of the form  \eqref{jordan form 2}. 
		
	\end{enumerate}
\end{remark}

Typically, to find the eigenvalues and eigenfunctions of the generator associated to a Markov chain is a  much more  challenging task than establishing duality relations. However, we have seen that the knowledge of the eigenfunctions leads to a full characterization of duality and/or self-duality functions. 
This is, indeed, the case of the example below, in which we  exploit the knowledge of eigenfunctions of two generators to characterize the family of self-duality and duality functions. 
\begin{example}[One-dimensional symmetric random walks on a finite grid]\label{section example 1d rw}

Let us introduce the symmetric random walk on $\Omega=\{1,\ldots, n \}$ \emph{reflected on the left} and \emph{absorbed on the right}. We describe the action  of the generator $L$ on functions $f : \Omega \to \R$ as
\beq \nonumber
L f(x) &=& (f(x+1)-f(x)) + (f(x-1)-f(x))\ ,\quad x \in \Omega \setminus \{1, n\}\ ,
\eeq
while for $x \in \{1,n\}$ we have
\beq \nonumber
L f(1)\ =\ 2(f(2)-f(1))\ ,\quad\quad Lf(n)\ =\ 0\ .
\eeq
Similarly, we denote by $\widehat L$ the generator of the symmetric random walk on $\Omega$ \emph{reflected on the right} and \emph{absorbed on the left}. Namely, 
\beq \nonumber
\widehat L f(x) &=& (f(x+1)-f(x)) + (f(x-1)-f(x))\ ,\quad x \in \Omega \setminus \{1, n\}\ ,
\eeq
and
\beq \nonumber
\widehat L f(1)\ =\ 0\ ,\quad \quad \widehat L f(n)\ =\ 2 (f(n-1)-f(n))\ .
\eeq
As an application of Theorem \ref{theorem jordan}, we prove the following dualities: \emph{self-duality} of $L$, \emph{self-duality} of $\widehat L$ and \emph{duality} between $L$ and $\widehat L$. The key is to explicitly find eigenvalues and eigenfunctions of the generators. Indeed, the eigenvalues $\{\lambda_1,\ldots, \lambda_n\}$  of $L$ and $\widehat L$ read as follows:
\beq \label{eigenvalues srw}
\lambda_1 \ =\ 0\ ,\quad \lambda_i =  2 (\cos(\theta_i)-1)\ ,\quad \theta_i\ =\ \frac{i-\frac{1}{2}}{n-1}\pi\ ,\quad i = 2,\ldots, n\ .
\eeq
The eigenfunctions $\{u_1,\ldots, u_n \}$ of $L$ are, for $x \in \Omega$, 
\beq \nonumber
u_1(x)\ =\ \frac{1}{\sqrt{n}}\ ,\quad u_i(x)\ =\ \frac{1}{\sqrt{n}}\cos(\theta_i (x-1))\ ,\quad i = 2,\ldots, n\ ,
\eeq
while the eigenfunctions $\{\widehat u_1,\ldots, \widehat u_n \}$ of $\widehat L$ are, for $x \in \Omega$, 
\beq \nonumber
\widehat u_1(\widehat x)\ =\ \frac{1}{\sqrt{n}}\ ,\quad \widehat u_i(\widehat x)\ =\ \frac{1}{\sqrt{n}} \sin(\theta_i(\widehat x-1))\ ,\quad i = 2,\ldots, n\ .
\eeq
Hence, we conclude the following:
\begin{enumerate}[label={(\alph*)}]
	\item \emph{Self-duality functions for $L$. }  For all values $a_1,\ldots, a_n \in \R$, the  function
	\begin{multline} \label{self-duality Lrw}
	D(x,y)\ =\ \sum_{i=1}^n a_i u_i(x) u_i(y)\ =\ \frac{a_1}{n}+\sum_{i=2}^n \frac{a_i}{n} \cos(\theta_i (x-1)) \cos(\theta_i (y-1))
	\end{multline}
	is a self-duality function for $L$ and all self-duality functions are of this form.
	\item \emph{Self-duality functions for $\widehat L$. } For all $a_1,\ldots, a_n \in \R$, 
	\begin{multline} \label{self-duality hat Lrw}
	\widehat D(\widehat x, \widehat y)\ =\ \sum_{i=1}^n a_i \widehat u_i(\widehat x) \widehat u_i(\widehat y)\ =\ \frac{1}{n} + \sum_{i=2}^n \frac{a_i}{n} \sin(\theta_i (\widehat x-1)) \sin(\theta_i(\widehat y-1))\ 
	\end{multline}
	is a self-duality function for $\widehat L$ and all self-duality functions are of this form. 
	\item \emph{Duality functions between $L$ and $\widehat L$. } For all $a_1,\ldots, a_n \in \R$,
	\beq \label{duality rws}
	D'(\widehat x, x)\ =\ \frac{a_1}{n} + \sum_{i=2}^n \frac{a_i}{n} \sin(\theta_i(\widehat x-1)) \cos(\theta_i(x-1))\ ,
	\eeq
	is a duality function between $L$ and $\widehat L$ and all duality functions are of this form.
	\qedsymbol 
\end{enumerate}

\end{example}

We can now provide an analogue of Proposition \ref{proposition cheaprev} beyond the reversible context. 
To fix notation, let $L$ be a generator on $\Omega$, with $|\Omega|=n$. Lacking reversibility, we have seen  that complex eigenvalues and generalized eigenfunctions of the generator may arise. However,  in the irreducible case, i.e., in case there exists a unique stationary measure $\mu > 0$ for which the adjoint of $L$ in $L^2(\mu)$, say $L^\dagger$, is itself a generator, a trivial duality relation between $L$ and $L^\dagger$ is available. Indeed, from the adjoint relation
\beq \nonumber
\langle L^\dagger f, g\rangle_{L^2(\mu)} &=& \langle f, L g \rangle_{L^2(\mu)}\ ,\quad f, g : \Omega \to \R\ ,
\eeq
it follows that the diagonal function $D : \Omega \times \Omega \to \R$ given by
\beq \label{Dcheap}
D(x,y) &=& \frac{\delta_{x,y}}{\mu(y)}\ ,\quad x,y \in \Omega\ ,
\eeq
is a duality function for $L^\dagger$, $L$. In analogy with \eqref{dcheap}, we refer to it as \emph{cheap duality function}, also $D=D_\cheap$. 

From Theorem \ref{theorem jordan}, the above duality tells us that, beside the fact that the generators $L$ and $L^\dagger$ are indeed similar as matrices,  the cheap duality function $D_\cheap$ in \eqref{Dcheap} should be represented in terms of functions $\{u_1,\ldots, u_n \}$ and $\{\widetilde u_1,\ldots, \widetilde u_n \}$, which, up to suitably reordering, are indeed the generalized eigenfunctions of $L$ and $L^\dagger$, respectively. 

As a consequence of the following lemma, which we use in the proof of Theorem \ref{theorem siegmund}, we obtain  that a relation of \emph{bi-orthogonality} w.r.t.\ $\mu$ among the generalized eigenfunctions of $L$ and those of $L^\dagger$  can be derived from the duality w.r.t.\ $D_\cheap$. For the proof, we refer back to the proof of Proposition \ref{proposition cheaprev}.

\begin{proposition}\label{lemma cheap}
	Let $L$ be a generator, $\mu$ a positive measure on $\Omega$ (not necessarily stationary for $L$) and let $L^\dagger$ be the adjoint operator  of $L$ in $L^2(\mu)$. 
	Let the spans of  the generalized eigenfunctions of $L$ and $L^\dagger$, say $\{u_1,\ldots, u_n \}$ and $\{\widetilde u_1,\ldots, \widetilde u_n \}$, both coincide with $L^2(\mu)$. Then the following statements are equivalent:
	\ben[label={(\roman*)}]
	\item \emph{Cheap duality from generalized eigenfunctions. } For $x, y \in \Omega$, 
	\beq \nonumber
	\sum_{i=1}^n \widetilde u_i( x) u_i(y) &=& \frac{\delta_{x, y}}{\mu(y)}\ .
	\eeq
	\item \emph{Bi-orthogonality of generalized eigenfunctions. } For all $i,j =1,\ldots, n$, 
	\beq \label{biorthogonality}
	\langle \widetilde u_i, u^\ast_j \rangle_\mu &=&  \sum_{x' \in \Omega} \widetilde u_i(x')u_j(x')\mu(x')\ =\  \delta_{i,j}\ .
	\eeq
	\een
	Two families $\{u_1^\ast,\ldots, u_n^\ast \}$, $\{\widetilde u_1,\ldots, \widetilde u_n \}$ satisfying condition \eqref{biorthogonality} are also said to be \emph{bi-orthogonal} w.r.t.\ the measure $\mu$.
\end{proposition}

\subsection{Intertwining relations, duality and generalized eigenfunctions}\label{subsection intertwining}
Symmetries of the generators or, more generally, \emph{intertwining relations} have proved to be useful in producing new duality relations from existing ones, e.g.\ cheap dualities \cite{carinci}, \cite{d1method}. Here, we analyze this technique and revisit \cite[Theorem 5.1]{d1method} from the point of view of  generalized eigenfunctions. 

\bt[Intertwining relations and duality]\label{theorem intertwining}
Let $L$, $\widetilde L$ and $\widehat L$ be three generators on $\Omega$, $\widetilde \Omega$ and $\widehat \Omega$ respectively. We assume that $L$ and $\widetilde L$ are \emph{intertwined}, i.e., there exists a linear operator $\varLambda : L^2(\Omega) \to L^2(\widetilde \Omega)$ such that, for all $f \in L^2(\Omega)$, we have
\beq \label{intertwining}
\widetilde L \varLambda f &=& \varLambda L f\ .
\eeq
Moreover, we assume that $L$ and $\widehat L$ are dual with duality function $D : \widehat \Omega \times \Omega \to \R$, i.e.,
\beq \nonumber
\widehat L_\eft D(\widehat x, x) &=& L_\ight D(\widehat x, x)\ .
\eeq
Then, the function $\varLambda_\ight D : \widehat \Omega \times \widetilde \Omega \to \R$ is a duality function for $\widetilde L$ and $\widehat L$, i.e., 
\beq \nonumber
\widehat L_\eft \varLambda_\ight D(\widehat x, \widetilde x) &=& \widetilde L_\ight \varLambda_\ight D(\widehat x, \widetilde x)\ .
\eeq
\et

\begin{proof}
	We observe that the \emph{intertwining operator} $\varLambda$ maps eigenspaces of $L$ to eigenspaces of $\widetilde L$.  
	More formally, if there exists a subset  $\{u^{(1)},\ldots, u^{(m)} \}$ of $L^2(\Omega)$ such that, for some $\lambda \in \C$, 
	\beq \label{intertwining geneigenfunctions}
	L u^{(1)}\ =\ \lambda u^{(1)}\ ,\quad  L u^{(k)}\ =\ \lambda u^{(k)} + u^{(k-1)}\ ,\quad k = 2,\ldots, m\ ,
	\eeq
	then, by \eqref{intertwining}, the subset $\{\varLambda u^{(1)},\ldots, \varLambda u^{(m)} \}$ in $L^2(\widetilde \Omega)$ satisfy the same identities as in \eqref{intertwining geneigenfunctions} up to replace $L$ by $\widetilde L$:
	\beq \label{intertwining geneigenfunctions 2}
	\widetilde L \varLambda u^{(1)}\ =\ \lambda \varLambda u^{(1)}\ ,\quad  \widetilde L \varLambda u^{(k)}\ =\ \lambda \varLambda u^{(k)} + \varLambda u^{(k-1)}\ ,\quad k = 2,\ldots, m\ .
	\eeq
%
By Theorem \ref{theorem jordan}, the duality function is given by
\beq \nonumber
D(\widehat x, x) &=& \sum_{i=1}^n \widehat u_i(\widehat x) u_i(x)\ ,
\eeq
where $\{u_1,\ldots,  u_n\}$, $\{\widehat u_1,\ldots, \widehat u_n \}$ are sets of (possibly generalized) eigenfunctions of $L$, $\widehat L$. Then, by applying the intertwining operator $\varLambda$ on the right variables, we obtain
\beq \nonumber
\varLambda_\ight D(\widehat x, \widetilde x) &=& \sum_{i=1}^n \widehat u_i(\widehat x) (\varLambda u_i)(\widetilde x)\ .
\eeq
We conclude from the considerations in \eqref{intertwining geneigenfunctions 2}, \eqref{intertwining geneigenfunctions} and Theorem \ref{theorem jordan}.
\end{proof}


%

Typical examples of intertwining relations occur when either $\varLambda$ is a \emph{symmetry} of a generator, i.e., $\widetilde L=L$ in \eqref{intertwining} (see e.g.\ \cite{carinci}) or when $\varLambda$ is a \emph{positive contractive operator} such that $\varLambda 1 = 1$, i.e., viewed as a matrix, it is a stochastic matrix from the space $\widetilde \Omega$ to $\Omega$ (see e.g.\ \cite{huilletmartinez}). A particular instance, which recovers the so-called \emph{lumpability},  of this last situation is when $\varLambda$ is a \textquotedblleft deterministic\textquotedblright\ stochastic kernel, i.e., induced by a map from $\widetilde \Omega$ to $\Omega$.

\subsection{Intertwining of exclusion processes}\label{subsection exclusion}
In this section we provide an application of Theorem \ref{theorem intertwining} above. Indeed, after finding suitable intertwining relations between a particular instance of the symmetric simple exclusion process and a generalized symmetric exclusion process, we obtain as in Theorem \ref{theorem intertwining} a large class of self-duality functions for the latter process from self-duality functions of the former. In what follows, we fix $\gamma \in \N$, a finite set $V$ of cardinality $|V|=m$ and a function $p : V \times V \to \R_+$ such that $p(x,x)=0$ for all $x \in V$.

The \emph{$\gamma$-ladder-$\SEP$} is the finite-state Markov process on $\widetilde \Omega = \{0,1 \}^{V \times \{1,\ldots, \gamma \}}$ with generator $\widetilde L$ acting on functions $\widetilde f : \widetilde \Omega \to \R$ as
\begin{multline*}
\widetilde L \widetilde f(\widetilde \eta)\ =\ \sum_{x,y \in V} p(x,y) \left[\sum_{a=1}^\gamma \sum_{b=1}^\gamma \widetilde \eta(x,a) (1-\widetilde \eta(y,b))\, (\widetilde f(\widetilde \eta^{(x,a),(y,b)}) - \widetilde f(\widetilde \eta))\right. \\
+\  \left.\widetilde \eta(y,b) (1-\widetilde \eta(x,a))\, (\widetilde f(\widetilde \eta^{(y,b),(x,a)})-\widetilde f(\widetilde \eta))\right]\ ,\quad \widetilde \eta \in \widetilde \Omega\ ,
\end{multline*}
where $\widetilde \eta^{(x,a), (y,b)}$ denotes the configuration obtained from $\widetilde \eta$ by removing a particle at position $(x,a)$ and placing it at $(y,b)$. As already mentioned, this process may be considered as a special case of a simple symmetric exclusion process on the set $\widetilde V_\gamma = V \times \{1,\ldots, \gamma\}$ where $\widetilde p : \widetilde V \times \widetilde V \to \R_+$ is such that
\beq \nonumber
\widetilde p((x,a),(y,b))\ =\ p(x,y)\ ,\quad (x,a) , (y,b) \in \widetilde V_\gamma\ .
\eeq

The $\SEP(\gamma)$ is the finite-state Markov process on $\Omega = \{0,\ldots, \gamma \}^V$ with generator $L$ acting on functions $f :  \Omega \to \R$ as
\begin{multline*} 
L f(\eta)\ =\ \sum_{x,y \in V} p(x,y)\ [\eta(x)(\gamma-\eta(y))\, (f(\eta^{x, y})-f(\eta)) \\
+\ \eta(y)(\gamma-\eta(x))\, (f(\eta^{y, x})-f(\eta))]\ ,\quad \eta \in \Omega\ .
\end{multline*}
It is well known (see e.g.\ \cite{kurchan2}) that $L$ and $\widetilde L$ are \emph{intertwined} via  a deterministic intertwining operator  $\varLambda : L^2(\Omega) \to L^2(\widetilde \Omega)$. The intertwining operator $\varLambda$ is  defined,  given the mapping $\pi : \widetilde \Omega \to \Omega$ such that 
\beq \nonumber
\pi(\widetilde \eta) &=& \left(| \widetilde \eta(1,\cdot)|,\ldots, | \widetilde \eta(n,\cdot)|\right)\ \in\ \Omega\ ,\quad |\widetilde \eta(x,\cdot)|\ :=\ \sum_{a=1}^\gamma \widetilde \eta(x,a)\ ,
\eeq
as acting on functions $f : \Omega \to \R$ as
\beq \nonumber
\varLambda f(\widetilde \eta) &=& f(\pi(\widetilde \eta))\ ,\quad \widetilde \eta \in \widetilde \Omega\ .
\eeq
The intertwining relation then reads, for all $f : \Omega \to \R$, as
\beq \nonumber
\widetilde L \varLambda f(\widetilde \eta) &=& \varLambda L f(\widetilde \eta)\ ,
\eeq
for $\widetilde \eta \in \widetilde \Omega$. Given any self-duality for $L$ with self-duality function $D(\xi,\eta)$, we can build a duality function, namely $D'(\xi,\widetilde \eta)=\varLambda_\ight D(\xi,\widetilde \eta)$ for $L$ and $\widetilde L$ and, furthermore, a self-duality function $D''(\widetilde \xi, \widetilde \eta)=\varLambda_\eft \varLambda_\ight D(\widetilde \xi, \widetilde \eta)$ for $\widetilde L$.

However, we ask whether there exists an \textquotedblleft inverse" intertwining relation, i.e., $\widetilde \varLambda : L^2(\widetilde \Omega) \to L^2(\Omega)$ such that, for $\widetilde f : \widetilde \Omega \to \R$,
\beq \label{inverse intertwining}
\widetilde \varLambda \widetilde L \widetilde f(\eta) &=& L \widetilde \varLambda \widetilde f(\eta)\ ,\quad \eta \in \Omega\ .
\eeq
In what follows, we say that $\widetilde \eta \in \widetilde \Omega$ is \emph{compatible} with $\eta \in \Omega$ or, shortly, $\widetilde \eta \sim \eta$, if $\pi(\widetilde \eta)=\eta$.
\begin{proposition} 
	The operator $\widetilde \varLambda : L^2(\widetilde \Omega) \to L^2(\Omega)$ defined as
	\beq \label{inverse intertwining operator}
	\widetilde \varLambda \widetilde f(\eta) &=& \left(\prod_{x \in V}\frac{1}{\binom{\gamma}{\eta(x)}}\right) \sum_{\widetilde \eta  \sim \eta} \widetilde f(\widetilde \eta)\ ,\quad \eta \in \Omega\ ,
	\eeq
	is the inverse intertwining in \eqref{inverse intertwining}. Moreover, the intertwining operator above is a \emph{stochastic} intertwining.
\end{proposition}
\begin{proof} Without loss of generality, we consider $V=\{x,y \}$.
	By expanding the l.h.s.\ of \eqref{inverse intertwining} with $\widetilde \varLambda$ as in \eqref{inverse intertwining operator}, we obtain four terms:
	$$
	\ell_1\ =\ -\frac{1}{\binom{\gamma}{\eta(x)}} \frac{1}{\binom{\gamma}{\eta(y)}} \sum_{\widetilde \eta \sim \eta} \sum_{a=1}^\gamma\sum_{b=1}^\gamma \widetilde \eta(x,a)(1-\widetilde \eta(y,b)) \widetilde f(\widetilde \eta) 
	$$
	$$
	\ell_2\ =\ \frac{1}{\binom{\gamma}{\eta(x)}} \frac{1}{\binom{\gamma}{\eta(y)}} \sum_{\widetilde \eta \sim \eta} \sum_{a=1}^\gamma\sum_{b=1}^\gamma \widetilde \eta(x,a)(1-\widetilde \eta(y,b)) \widetilde f(\widetilde \eta^{(x,a),(y,b)})
	$$
	$$
	\ell_3\ =\ -\frac{1}{\binom{\gamma}{\eta(x)}} \frac{1}{\binom{\gamma}{\eta(y)}} \sum_{\widetilde \eta \sim \eta} \sum_{a=1}^\gamma\sum_{b=1}^\gamma \widetilde \eta(y,b)(1-\widetilde \eta(x,a)) \widetilde f(\widetilde \eta)
	$$
	$$
	\ell_4\ =\ \frac{1}{\binom{\gamma}{\eta(x)}} \frac{1}{\binom{\gamma}{\eta(y)}} \sum_{\widetilde \eta \sim \eta} \sum_{a=1}^\gamma\sum_{b=1}^\gamma \widetilde \eta(y,b)(1-\widetilde \eta(x,a)) \widetilde f(\widetilde \eta^{(y,b),(x,a)})\ .
	$$
	By doing the same thing with the r.h.s., we obtain:
	$$
	r_1\ =\ - \frac{1}{\binom{\gamma}{\eta(x)}}\frac{1}{\binom{\gamma}{\eta(y)}} \eta(x)(\gamma-\eta(y)) \sum_{\widetilde \eta \sim \eta} \widetilde f(\widetilde \eta) 
	$$
	$$
	r_2\ =\ \frac{1}{\binom{\gamma}{\eta(x)-1}}\frac{1}{\binom{\gamma}{\eta(y)+1}} \eta(x)(\gamma-\eta(y)) \sum_{\widetilde \eta \sim \eta^{x,y} } \widetilde f(\widetilde \eta)
	$$
	$$
	r_3\ =\ - \frac{1}{\binom{\gamma}{\eta(x)}}\frac{1}{\binom{\gamma}{\eta(y)}} \eta(y)(\gamma-\eta(x)) \sum_{\widetilde \eta \sim \eta} \widetilde f(\widetilde \eta) 
	$$
	$$
	r_4\ =\ \frac{1}{\binom{\gamma}{\eta(x)+1}}\frac{1}{\binom{\gamma}{\eta(y)-1}} \eta(y)(\gamma-\eta(x)) \sum_{\widetilde \eta \sim \eta^{y,x} } \widetilde f(\widetilde \eta)\ .
	$$
	Note that $\ell_1 = r_1$ because, for all $\widetilde \eta \sim \eta$, 
	\beq \nonumber
	\sum_{a=1}^\gamma\sum_{b=1}^\gamma \widetilde \eta(x,a) (1-\widetilde \eta(y,b)) &=& \eta(x) (\gamma-\eta(y))\ ,
	\eeq
	and similarly for $\ell_3=r_3$. For $\ell_2=r_2$ it is enough to verify that, for each $\widetilde \eta_\ast \sim \eta^{x,y}$, 
	\beq \nonumber
	\sum_{\widetilde \eta \sim \eta} \sum_{a=1}^\gamma \sum_{b=1}^\gamma \widetilde \eta(x,a)(1-\widetilde \eta(y,b))  \1\{\widetilde \eta^{(x,a),(y,b)} = \widetilde \eta_\ast \} &=& (\eta(y)+1)(\gamma-\eta(x)+1)\ .
	\eeq
	This last identity indeed holds, as the configurations $\widetilde \eta \sim \eta$ can be obtained from  $\widetilde \eta_\ast$  by picking one of the $\eta(y)+1$ particles on $y \in V$ and putting it back on one of the $\gamma-\eta(x)+1$ holes of $x \in V$. Analogously for $\ell_4=r_4$.	
\end{proof}

As a consequence of this proposition, by starting from self-duality of the $\gamma$-ladder-$\SEP$, we can produce duality functions for $\widetilde L$ and $L$ and self-duality functions for $L$.
We use the following result of \cite[Theorem 2.8]{schutz lecture notes} to obtain a large class of \textquotedblleft factorized\textquotedblright\ self-duality functions for $\widetilde L$.
\begin{theorem}[\cite{schutz lecture notes}]
	The \emph{simple symmetric exclusion process} $\{\widetilde \eta_t,\ t \geq 0 \}$ on the vertex set $V \times \{1,\ldots, \gamma \}$ is self-dual w.r.t.\  the duality function
	\beq \label{self-duality SSEP}
	\widetilde D(\widetilde \xi, \widetilde \eta) &=& \prod_{(x,a) \in V\times \{1,\ldots, \gamma \}} (\alpha + \beta \widetilde \eta(x,a))^{\epsilon + \delta \widetilde \xi(x,a)}\ ,\quad \widetilde \xi, \widetilde \eta \in \widetilde \Omega \ ,
	\eeq
	for all $\alpha, \beta, \epsilon$ and $\delta \in \R$.
\end{theorem}
Now, we apply the intertwining operator $\widetilde \varLambda$ first on the right and then on the left variables of $\widetilde D$ above.  
\begin{theorem}
	All self-duality functions for $\SEP(\gamma)$ derived from self-duality functions of $\gamma$-ladder-$\SEP$ as in \eqref{self-duality SSEP} are all in factorized form, i.e., 
	\beq \nonumber
	D(\xi,\eta)\ =\ \widetilde \varLambda_\eft \widetilde \varLambda_\ight \widetilde D( \xi, \eta) &=& \prod_{x \in V} d_x^{\alpha, \beta, \epsilon, \delta}(\xi(x), \eta(x))\ .
	\eeq
	Moreover, the \emph{single-site self-duality functions} $d_x^{\alpha,\beta, \epsilon, \delta}(k,n)$, for $k,n \in \{0,\ldots, \gamma\}$, are in one of the following forms: either the \emph{classical} polynomials
	\beq \nonumber
	d^{0,\beta,0,\delta}_x(k,n) &=& (\beta^\delta)^k\frac{(\gamma-k)!}{\gamma!} \frac{n!}{(n-k)!} \1\{n \geq k \}\ ,
	\eeq
	the \emph{orthogonal} polynomials
	\beq \nonumber
	d_x^{\alpha,\beta, \epsilon, \delta}(k,n) &=& (-1)^{\delta k} \alpha^{\epsilon \gamma- \epsilon n+\delta k} (\alpha+\beta)^{\epsilon n} \pFq{2}{1}{-k,-n}{-\gamma}{1-\left( 1+\frac{\beta}{\alpha}\right)^\delta}\ ,
	\eeq
	or other  degenerate functions:
	\beq \nonumber
	d_x^{\alpha, \beta, \epsilon, 0}(k,n) &=& (\alpha+\beta)^{\epsilon n} \alpha^{\epsilon (\gamma-n)}\\
	\nonumber
	d^{0,\beta,\epsilon,\delta}_x(k,n) &=& \beta^{\epsilon \gamma + \delta k} \1\{n= \gamma \}
	\\
	\nonumber
	d_x^{\alpha,0,\epsilon,\delta}(k,n) &=& \alpha^{\epsilon \gamma + \delta k}
	\\
	\nonumber
	d^{\alpha, -\alpha, \epsilon, \delta}_x(k,n) &=& \alpha^{\epsilon \gamma + \delta k} \1\{n = 0 \}\ .
	\eeq
\end{theorem}
\begin{proof}
	First thing to note is that the factorized structure of $D$ is preserved under $\widetilde \varLambda$. Indeed, if we use the notation
	\beq \nonumber
	\mathtt d(k,n) &=& (\alpha+\beta n)^{\epsilon+ \delta k}\ ,
	\eeq
	then 
	\beq \nonumber
	\widetilde \varLambda_\ight D(\widetilde \xi, \eta) &=&   \prod_{x \in V} \left(\frac{1}{\binom{\gamma}{\eta(x)}} \sum_{\widetilde \eta(x,\cdot)\sim \eta(x)} \prod_{a=1}^\gamma \mathtt d(\widetilde \xi(x,a), \widetilde \eta(x,a)) \right)\ .
	\eeq
	Hence we compute only what is inside the parenthesis (which will see does depend on $\widetilde \xi(x,\cdot)$ only through $|\widetilde \xi(x,\cdot)|$):
	\begin{multline} \label{single-site general}
	d^{\alpha, \beta, \epsilon, \delta}_x(\xi(x),\eta(x))) \\
	=\
	(\alpha+\beta)^{\epsilon \eta(x)} \alpha^{\epsilon(\gamma-\eta(x))}
	\frac{1}{\binom{\gamma}{\eta(x)}}  \sum_{\widetilde \eta(x,\cdot) \sim \eta(x)} \prod_{a=1}^\gamma (\alpha+\beta \widetilde \eta(x,a))^{\delta \widetilde \xi(x,a)}\ .
	\end{multline}
	The last summation
	\beq \nonumber
	\frac{1}{\binom{\gamma}{\eta(x)}}  \sum_{\widetilde \eta(x,\cdot) \sim \eta(x)} \prod_{a=1}^\gamma (\alpha+\beta \widetilde \eta(x,a))^{\delta \widetilde \xi(x,a)}
	\eeq
	clearly does not depend on $\widetilde \xi(x,\cdot)$ but only on $\xi(x)=|\widetilde \xi(x,\cdot)|$ and equals
	\beq \label{quantity}
	\frac{1}{\binom{\gamma}{\eta(x)}} \sum_{\ell=0}^{\xi(x)} \binom{\xi(x)}{\xi(x)-\ell} \binom{\gamma-\xi(x)}{\eta(x)-(\xi(x)-\ell)} (\alpha+\beta)^{\delta (\xi(x)-\ell)} \alpha^{\delta \ell} \ .
	\eeq
	If $\delta=0$, this last expression in \eqref{quantity} by Chu-Vandermonde identity equals $1$, hence
	\beq \nonumber
	d^{\alpha, \beta, \epsilon, 0}_x(\xi(x),\eta(x)) &=&  (\alpha+\beta)^{\epsilon \eta(x)} \alpha^{\epsilon (\gamma-\eta(x))}\ .
	\eeq
	If  $\delta \neq 0$ and $\alpha= 0$,  expression \eqref{quantity} rewrites as
	\begin{multline*} 
	\frac{1}{\binom{\gamma}{\eta(x)}} \binom{\gamma-\xi(x)}{\eta(x)-\xi(x)} \beta^{\delta \xi(x)} \1\{\eta(x) \geq \xi(x) \} \\ =\ (\beta^\delta)^{\xi(x)}\frac{(\gamma-\xi(x))!}{\gamma!} \frac{\eta(x)!}{(\eta(x)-\xi(x))!} \1\{\eta(x) \geq \xi(x) \}\ ,
	\end{multline*}
	and hence, for $\epsilon = 0$,  \eqref{single-site general} becomes
	\beq \nonumber
	d^{0,\beta,0,\delta}_x(\xi(x),\eta(x))\ =\ (\beta^\delta)^{\xi(x)}\frac{(\gamma-\xi(x))!}{\gamma!} \frac{\eta(x)!}{(\eta(x)-\xi(x))!} \1\{\eta(x) \geq \xi(x) \}\ ,
	\eeq
	i.e., the classical single-site self-duality functions, while, for $\epsilon \neq 0$, 
	\beq \nonumber
	d^{0,\beta,\epsilon,\delta}_x(\xi(x),\eta(x)) &=& \beta^{\epsilon \gamma + \delta \xi(x)} \1\{\eta(x) = \gamma \}\ .
	\eeq
	If $\delta \neq 0$ and $\alpha \neq 0$ and $\beta = 0$, then again we get some trivial:
	\beq \nonumber
	d_x^{\alpha,0,\epsilon,\delta}(\xi(x),\eta(x)) &=& \alpha^{\epsilon \gamma + \delta \xi(x)}\ .
	\eeq
	The most interesting case is when $\delta \neq 0$, $\alpha \neq 0$, $\beta \neq 0$ and $\alpha \neq - \beta$. In this case the quantity in $\eqref{quantity}$ equals
	\beq \nonumber
	(\alpha+\beta)^{\delta \xi(x)} \frac{1}{\binom{\gamma}{\eta(x)}} \sum_{\ell=0}^{\xi(x)} \binom{\xi(x)}{\xi(x)-\ell} \binom{\gamma-\xi(x)}{\eta(x)-(\xi(x)-\ell)} \left(\frac{\alpha}{\alpha+\beta} \right)^{\delta \ell}\ ,
	\eeq
	which rewrites, by using two known relations in \cite[p.\ 51]{nikiforov},  as
	\beq \nonumber
	(-\alpha)^{\delta \xi(x)}\pFq{2}{1}{-\xi(x),-\eta(x)}{-\gamma}{1-\left( 1+\frac{\beta}{\alpha}\right)^\delta}\ ,
	\eeq
	leading to
	\begin{small}
	\begin{multline*}
	d_x^{\alpha,\beta, \epsilon, \delta}(\xi(x), \eta(x)) \\
	=\ (-1)^{\delta \xi(x)} \alpha^{\epsilon \gamma- \epsilon \eta(x)+\delta \xi(x)} (\alpha+\beta)^{\epsilon \eta(x)} \pFq{2}{1}{-\xi(x),-\eta(x)}{-\gamma}{1-\left( 1+\frac{\beta}{\alpha}\right)^\delta}\ , 
	\end{multline*}
	\end{small}i.e., we recover the \emph{orthogonal polynomial single-site self-duality functions} for the $\SEP(\gamma)$, namely families of Kravchuk polynomials. 
	If $\alpha=-\beta$, then we have
	\beq \nonumber
	d^{\alpha, -\alpha, \epsilon, \delta}_x(\xi(x), \eta(x)) &=& \alpha^{\epsilon \gamma + \delta \xi(x)} \1\{\eta(x) = 0 \}\ .
	\eeq
\end{proof}

%
%

\section{Siegmund duality}\label{section Siegmund}
This connection between duality functions and eigenfunctions enables us to recover another special instance of duality, the so-called \emph{Siegmund duality}. Siegmund duality, which arises in the context of totally ordered state spaces $\Omega = \hat \Omega$, was first established by Siegmund \cite{siegmund} for pairs of absorbed/reflected-at-$0$ processes on the positive real line and on the positive integers. Further applications and  generalizations of Siegmund dualities were studied by many authors, see for instance \cite{kolokolstov}, \cite{phdthesissiegmund}, \cite{liggett}.

 What we focus here on is a finite-context characterization of Siegmund duality already obtained via an intertwining relation in \cite{huilletmartinez}. However,  by using a representation of duality in terms of generalized eigenfunctions of the generators,  the characterization result of Siegmund duality that we obtain, besides simplifying the proof of an analogous result in \cite[Theorem 3]{siegmund}, adds  spectral information to the proof in \cite{huilletmartinez}. 
 
 Moreover, as Siegmund duality can be seen as a full-rank duality between two processes, cf.\ Theorem \ref{theorem jordan}, a spectral approach provides a strategy to find other duality relations in the presence of  Siegmund duality.
\subsection{Characterization of Siegmund duality}
On the totally ordered state space $\Omega=\{1,\ldots, n \}$, two generators  $L$, $\widehat L$ are said to be \emph{Siegmund dual} if
\beq \label{Siegmund duality}
\widehat L\eft D_\s(x,y) &=&  L_\ight D_\s(x, y)\ ,
\eeq
with duality function  $D_\s : \Omega \times \Omega \to [0,1]$ given by
\beq \label{Siegmund duality function}
D_\s(x,y) &=& \1\{x \geq y \}\ .
\eeq
Note that the duality relation \eqref{Siegmund duality} with duality function $D_\s$ \eqref{Siegmund duality function} reads out
\beq \label{Siegmund duality 2}
\sum_{x'=y}^n \widehat L(x,x') &=& \sum_{y'=1}^x L(y,y')\ .
\eeq
From \eqref{Siegmund duality 2}, a  \emph{necessary} relation between two Siegmund dual generators $L$ and $\widehat L$ reads as follows:
\beq \label{Lsiegmund}
L(y,x)\ &=& \sum_{x'=y}^n \widehat L(x,x')-\widehat L(x-1,x')\ ,\quad x , y \in \Omega\ ,
\eeq
with the convention $\widehat L(0,\cdot) = 0$. As \eqref{Lsiegmund} implies \eqref{Siegmund duality 2}, this condition is indeed also \emph{sufficient}.
\begin{remark}[Sub-generators and monotonicity] \label{remark siegmund generator} 
If we  require that only $\widehat L$ is a generator,  the operator $L$ as defined in \eqref{Lsiegmund} is not necessarily a generator. However, the following implications hold:
\begin{enumerate}[label={(\alph*)}]
	\item If $\widehat L$ is a generator and $L(y,x) \geq 0$ for all $x \neq y$, then $L$ is a  \emph{sub-generator} on $\Omega$, i.e.,
\beq \label{sub-generator}
L(y,x)\ \geq \ 0\ ,\ x \neq y\ \quad \text{and}\quad \sum_{x =1}^n L(y,x)\ \leq\ 0\ ,\ y \in \Omega\ .
\eeq
The proof goes as follows: 
\begin{align*} \nonumber
\sum_{x =1}^n L(y,x)\ &=\ \sum_{x'=y}^n \sum_{x=1}^n \widehat L(x,x') - \widehat L(x-1,x')\\ &=\ \sum_{x'=y}^n \widehat L(n, x')\ \leq\ \sum_{x'=1}^n \widehat L(n,x')\ =\ 0\ ,
\end{align*}
where we used \eqref{Lsiegmund} in the first equality and the last inequality is a consequence of $\widehat L$ being a generator.  
\item Note that, by \cite[Theorem 2.1]{keilson}, 
\beq \label{monotonicity condition}
 \sum_{x'= y}^n \widehat L(x,x') - \widehat L(x-1,x')\ \geq\ 0\ ,\quad x \neq y\ ,
\eeq 
is equivalent to require that the continuous-time Markov chain with generator $\widehat L$ is \emph{monotone} (see \cite{liggett}). 
\end{enumerate}
As a consequence, $L$ is a sub-generator if and only if $\widehat L$ is associated to a monotone process on $\Omega$.
\end{remark}
In the following theorem, we study the relation between eigenfunctions of Siegmund dual  (sub-)generators and how the Siegmund duality function $D_\s$ in \eqref{Siegmund duality function} is constructed from the eigenfunctions. 

\bt \label{theorem siegmund}

\begin{enumerate}[label={(\roman*)}]
	\item Let $L$ and $\widehat L$ be Siegmund dual (sub-)generators in the sense of \eqref{Siegmund duality}. If $\widehat w$ is a $k$-th order generalized eigenfunction of $\widehat L^\mathsf{T}$ associated to eigenvalue $\lambda$, then 
	\beq \label{eigenfunctionrel}
	u(x)\ =\ \sum_{y=x}^n \widehat w(y)\ ,\quad x \in \Omega\ ,
	\eeq
	is a $k$-th order generalized eigenfunction  of $L$ associated to the eigenvalue $\lambda$.
	\item In the same context as in item (i), given a set $\{\widehat w_1,\ldots, \widehat w_n \}$ of (generalized) eigenfunctions of $\widehat L^\mathsf{T}$ whose span coincides with $L^2(\Omega)$, if $\{\widehat u_1,\ldots, \widehat u_n \}$ are (generalized) eigenfunctions of $\widehat L$ such that
	\beq \label{normalized}
	\langle \widehat w_i, \widehat u_j^\ast \rangle &=& \sum_{x = 1}^n \widehat w_i(x) \widehat u_j(x)\ =\ \delta_{i,j}\ ,
	\eeq 
	and $\{u_1,\ldots, u_n \}$ are defined in terms of $\{\widehat w_1,\ldots, \widehat w_n \}$ as in \eqref{eigenfunctionrel}, 
	then the function
	\beq \nonumber
	D(x,y) &=& \sum_{i=1}^n \widehat u_i(x) u_i(n)\ ,\quad x, y \in \Omega\ ,
	\eeq
	is the Siegmund duality function $D_\s$. 
	\item Let $L$ and $\widehat L$ be (sub-)generators on $\Omega$. If for any $k$-th order generalized eigenfunction $\widehat w$ of $\widehat L^\mathsf{T}$ associated to eigenvalue $\lambda$, $u$ as defined in \eqref{eigenfunctionrel} is a $k$-th order generalized eigenfunction of $L$ associated to the same eigenvalue $\lambda$, then $L$ and $\widehat L$ are Siegmund dual and $D_\s$ is obtained as in item (ii).
\end{enumerate}

\et
\begin{proof}
Let $\widehat w$ and $u$ be as in item (i). Then,
\beq \nonumber
\sum_{x=1}^n L(y,x) u(x) &=& \sum_{x=1}^n \left(\sum_{x'=y}^n \widehat L(x,x')-\widehat L(x-1,x')\right) u(x)\\
\nonumber
&=&  \sum_{x'=y}^n \sum_{x=1}^n \left( \widehat L^\mathsf{T}(x',x) u(x) - \widehat L^\mathsf{T}(x',x-1)u(x) \right)\ ,
\eeq
which, by noting that $\widehat w(n) = u(n)$, reads as
\beq \nonumber
\sum_{x'=y}^n \sum_{x=1}^n \widehat L^\mathsf{T}(x',x) \widehat w(x) &=& \sum_{x'=y}^n \lambda \widehat w(x')\ =\ \lambda \sum_{x'=y}^n \widehat w(x')\ =\ \lambda u(y)\ ,
\eeq
thus, $u$ is eigenfunction with eigenvalue $\lambda$. For the generalized eigenfunctions, the proof follows the same line.

For item (ii) and (iii), from the sets $\{\widehat w_1,\ldots, \widehat w_n \}$ and $\{u_1,\ldots, u_n \}$ of generalized eigenfunctions of $\widehat L^\mathsf{T}$ and $L$ related as in \eqref{eigenfunctionrel},  by Theorem  \ref{theorem jordan} the function
\begin{multline} \label{Dpresiegmund}
D(x,y)\ =\ \sum_{i=1}^n \widehat u_i(x) u_i(y)\ =\ \sum_{i=1}^n \widehat u_i(x) \sum_{x'=y}^n \widehat w_i(x')\ =\ \sum_{x'=y}^n \sum_{i=1}^n \widehat u_i(x) \widehat w_i(x')\ 
\end{multline}
is a full-rank duality for $L$ and $\widehat L$.
By Proposition \ref{lemma cheap} and condition \eqref{normalized}, by passing to the conjugates, we obtain 
\beq \nonumber
\sum_{i=1}^n \widehat u_i(x) \widehat w_i(x') &=& \delta_{x,x'}\ ,
\eeq
and hence the function $D(x,y)$ in \eqref{Dpresiegmund} writes as
\beq \nonumber
D(x,y) &=& \sum_{x'=y}^n \delta_{x,x'}\ =\ \1\{x \geq y \}\ =\ D_\s(x,y)\ .
\eeq
\end{proof}

In this final example, by using item (iii) of Theorem \ref{theorem siegmund}, we show how to obtain Siegmund duality from the knowledge of eigenvalues and eigenfunctions of (sub-)generators. The example we consider here concerns two symmetric random walks on $\Omega=\{1,\ldots, n \}$.

\begin{example}[Blocked vs absorbed random walks on a finite grid]\label{section example siegmund}
 The first symmetric nearest-neighbor random walk is \emph{blocked at the boundaries}, namely the generator $\widehat L$ is described, for $f :  \Omega \to \R$, as
\beq \nonumber
\widehat L f(x) &=& (f(x+1)-f(x)) + (f(x-1)-f(x))\ ,\quad x \in \Omega \setminus \{1,n \}\ ,
\eeq
and, on the boundaries, 
\beq \nonumber
\widehat Lf(1)\ =\ f(2)-f(1)\ ,\quad \quad \widehat Lf(n)\ =\ f(n-1)-f(n)\ .
\eeq
The second random walk is \emph{absorbed at the boundaries}, i.e., it is a sub-Markov process on $\Omega=\{1,\ldots, n \}$ with sub-generator $L$ which acts on functions $f : \Omega \to \R$ as
\beq \nonumber
L f(x) &=& (f(x+1)-f(x)) + (f(x-1)-f(x))\ ,\quad x \in \Omega \setminus \{1,n\}\ ,
\eeq
and
\beq \nonumber
L f(1)\ =\ 0 \ ,\quad \quad 
 Lf(n)\ =\ f(n-1)-2 f(n)\ ,
 \eeq
 i.e.\ $x=1$ is an absorbing point, while at $x=n$ the random walk either jumps to the left at rate $1$ or \textquotedblleft exits the system\textquotedblright\ at rate $1$. 
%

To explicitly obtain eigenfunctions and eigenvalues in this setting we use the following \emph{ansatz}:
\beq \nonumber
f_{a,b,c,\theta}(x) &=& a \cos(\theta x+c) + b \sin(\theta x+c)\ ,\quad x \in \Omega\ , 
\eeq
where $a$, $b$, $c$ and $\theta \in \R$ are the parameters to be determined. Regarding the eigenvalues $\{\lambda_1,\ldots, \lambda_n \}$, in both cases we have
\beq \nonumber
\lambda_1\ =\ 0\ ,\quad \lambda_i\ =\ 2 (\cos(\theta_i)-1)\ ,\quad \theta_i\ =\ \frac{i-1}{n}\pi\ ,\quad i=2,\ldots, n\ .
\eeq
Hence, all eigenvalues are distinct. The  eigenfunctions $\{\widehat u_1,\ldots, \widehat u_n \}$ of $\widehat L$ are, for $x \in \{1,\ldots, n \}$ and $i=2,\ldots, n$, 
\beq \nonumber
\widehat u_1(x)\ =\ \frac{1}{\sqrt{n}}\ ,
\eeq
and
\begin{multline*}
\widehat u_i(x)\ =\ \frac{1}{\sqrt{n (1-\cos(\theta_i))}} (-\sin(\theta_i)\cos(\theta_i (x-1)) + (1-\cos(\theta_i)) \sin(\theta_i (x-1)))\ .
\end{multline*}
The eigenfunctions $\{u_1,\ldots, u_n \}$ of $L$ are given, for $x \in \{1,\ldots, n \}$ and $i=2,\ldots, n$, by
\beq \nonumber
u_1(x)\ =\ \frac{n+1-x}{\sqrt{n}}\ ,\quad u_i(x)\ =\ \frac{1}{\sqrt{n(1-\cos(\theta_i))}} \sin(\theta_i (x-1))\ .
\eeq
Hence, we note that:
\begin{enumerate}[label={(\alph*)}]
	\item By Theorem \ref{theorem jordan}, $L$ and $\widehat L$ are dual and any duality function is of the form 
	\beq \label{generaldualityrws}
	D(x,y) &=& \sum_{i=1}^n a_i \widehat u_i(x) u_i(y)\ ,
	\eeq
	for $a_1,\ldots, a_n \in \R$.
	\item By denoting by $\nu$ the counting measure on $\Omega=\{1,\ldots, n \}$,  the generator $\widehat L$ is self-adjoint in $L^2(\nu)$ and is, as a matrix, symmetric, i.e., $\widehat L^\mathsf{T} = \widehat L$. As a consequence,   $\{\widehat u_1,\ldots, \widehat u_n  \}$ are eigenfunctions of both $\widehat L$ and $\widehat L^\mathsf{T}$.
	\item For all $i=1,\ldots, n$, 
	\beq \nonumber
	u_i(x) &=& \sum_{y=x}^n \widehat u_i(y)\ ,\quad x \in \Omega\ ,
	\eeq
	i.e., the eigenfunctions $\{u_1,\ldots, u_n \}$ are related to $\{\widehat u_1,\ldots, \widehat u_n \}$ as in \eqref{eigenfunctionrel}.
	\item The eigenfunctions $\widehat u_1,\ldots, \widehat u_n$ are normalized in $L^2(\nu)$, i.e., for all $i,j =1,\ldots, n$,
	\beq \nonumber
	\langle \widehat u_i, \widehat u_j \rangle_{L^2(\nu)} &=&\delta_{i,j}\ .
	\eeq
\end{enumerate}
As a consequence, by Theorem \ref{theorem siegmund}, for the choice $a_1=\ldots =a_n=1$, the duality function $D(x,y)$ in \eqref{generaldualityrws} is the Siegmund duality function $D_\s(x,y)$ in \eqref{Siegmund duality function}, namely, for all $x, y \in \Omega$, 
\begin{small}
\begin{multline*} \nonumber
 \frac{n+1-y}{n} \\+\ \sum_{i=2}^n \frac{ \sin(\theta_i (y-1))}{n(1-\cos(\theta_i))}\,(-\sin(\theta_i) \cos(\theta_i(x-1))+ (1-\cos(\theta_i)) \sin(\theta_i (x-1)))\\ =\ \1\{x \geq y\}\ .
\end{multline*}
\end{small}

As a final remark, we note that, by adding the cemetery state $\Delta=\{n+1 \}$ accessible at rate $1$ only from the state $\{n \}$, the absorbed sub-Markov random walk associated to $L$ becomes a proper Markov process with $\{1 \}$ and $\{n+1 \}$ as absorbing states. If we denote by $L^\text{ext}$ the generator on the extended space $\Omega \cup \Delta$, it follows that the eigenvalues of $L^\text{ext}$ remain unchanged, while the new eigenfunctions $\{u^\text{ext}_1,\ldots, u^\text{ext}_n, u^\text{ext}_{n+1} \}$ are such that
\beq \nonumber
u^\text{ext}_{n+1}(x) &=& 1\ ,\quad x \in \Omega \cup \Delta\ ,
\eeq
and, for all $i =1,\ldots, n$, 
\beq \nonumber
 u^\text{ext}_i(n+1)\ =\ 0\ ,\quad \quad u^\text{ext}_i(x)\ =\ u_i(x)\ ,\quad x \in \Omega\ .
\eeq
Hence, the function
\beq \nonumber
D^\text{ext}_\s(x,y) &=& \sum_{i=1}^n \widehat u_i(x) u^\text{ext}_i(y)\ ,\quad x \in \Omega\ ,\quad y \in \Omega \cup \Delta\ ,
\eeq
equals $\1\{x \geq y \}$.
\qedsymbol
\end{example}


\subsubsection*{Acknowledgments}
The authors thank  \emph{Institut Henri Poincar\'{e}}, where part of this work was done, for very kind hospitality.
F.S.\ acknowledges
NWO for financial support via the TOP1 grant 613.001.552. The same author is indebted to G.\ Carinci for fruitful discussions.

\end{document}